\documentclass[a4paper,12pt]{amsart}

\usepackage[english]{babel}
\usepackage[utf8]{inputenc}
\allowdisplaybreaks  % permite cortes de página en fórmulas

\usepackage{amssymb,xypic,amscd}

\usepackage{mathtools}
\usepackage{etoolbox}

\usepackage[colorlinks,linkcolor=Red3,citecolor=Red3,urlcolor=Cyan3]{hyperref}

\usepackage{enumitem}

\usepackage{parskip}

\usepackage[svgnames,x11names,dvipsnames]{xcolor}

\usepackage{bm}

\newtheorem{definition}[equation]{Definition}
\newtheorem{theorem}[equation]{Theorem}

\newtheorem{corollary}[equation]{Corollary}

\newtheorem{proposition}[equation]{Proposition}

\theoremstyle{remark}
\newtheorem{remark}[equation]{Remark}
\theoremstyle{remark}
\newtheorem{example}{Example}

\numberwithin{equation}{section} %\numberwithin{defn}{section}

\newcommand\res{\operatorname{Res}}

\renewcommand\deg{\operatorname{deg}}

\DeclarePairedDelimiterXPP\globalsymbol[1]{}{\langle}{\rangle}{}{\ifblank{#1}{\cdot,\cdot}{#1}}
\newcommand{\normSymbol}{N_{k(x)/k}}
\DeclarePairedDelimiterXPP\norm[1]{\normSymbol}{\lparen}{\rparen}{}{\ifblank{#1}{\cdot}{#1}}
\DeclarePairedDelimiterXPP\localsymbol[1]{}{\lparen}{\rparen}{}{\ifblank{#1}{\cdot,\cdot}{#1}}
\DeclarePairedDelimiterXPP\globalCommutator[1]{}{\lbrace}{\rbrace}{\raisebox{0pt}[1ex][1ex]{$^{\mathbb{A}_{X}}_{\mathbb{A}^{+}_{X}}$}}{\ifblank{#1}{\cdot,\cdot}{#1}}
\DeclarePairedDelimiterXPP\localCommutator[1]{}{\lbrace}{\rbrace}{^{K_{x}}_{\widehat{\O}_{X,x}}}{\ifblank{#1}{\cdot,\cdot}{#1}}

\renewcommand{\O}{{\mathcal{O}}}

\newcommand{\bC}{{\mathbb{C}}}
\newcommand{\bR}{{\mathbb{R}}}
\newcommand{\bZ}{{\mathbb{Z}}}
\newcommand{\bN}{{\mathbb{N}}}

\newcommand{\bQ}{{\mathbb{Q}}}

\newcommand{\cS}{{\mathcal{S}}}
\newcommand{\vand}{{\operatorname{V}}}
\renewcommand{\norm}{{\operatorname{N}}}
\renewcommand\tilde{\widetilde}

\makeatletter
\@namedef{subjclassname@2020}{%
  \textup{2020} Mathematics Subject Classification}
\makeatother

\title[The zeta function of a recurrence sequence]{The zeta function of a recurrence sequence of arbitrary degree}

\author[Á. Serrano~Holgado]{Álvaro Serrano Holgado}
\email{Alvaro\_ Serrano@usal.es}

\author[L.~M.~Navas~Vicente]{Luis Manuel Navas Vicente}
\email{navas@usal.es}

\thanks{Research of both authors supported by grant PGC2018-099599-B-I00 of the MICINN (Spain) and that of the second author also by grant PID2021-124332NB-C22 of the MICINN (Spain).}
\address{Departamento de Matem\'aticas and IUFFYM, Universidad de
Salamanca,  Plaza de la Merced 1-4
        \\
        37008 Salamanca. Spain.
        \\
         Tel: +34 923294460. 
}



\begin{document}

\maketitle

\begin{abstract}
We consider a Dirichlet series $\sum_{n=1}^{\infty} a_{n}^{-s}$ where $a_{n}$ satisfies a linear recurrence of arbitrary degree with integer coefficients. Under suitable hypotheses, we prove that it has a meromorphic continuation to the complex plane, giving explicit formulas for its pole set and residues, as well as for its finite values at negative integers, which are shown to be rational numbers. To illustrate the results, we focus on some concrete examples which have also been studied previously by other authors.
\end{abstract}

%%%%%%%%%%%%%%%%%%%%%%%%%%%%%%%%%%%%%%%%%%%%%%%%%%%%%%%%%%%%%%%%%%%%%%%%%%%%%%%%%%%
\section{Introduction}
\label{sec:intro}
%%%%%%%%%%%%%%%%%%%%%%%%%%%%%%%%%%%%%%%%%%%%%%%%%%%%%%%%%%%%%%%%%%%%%%%%%%%%%%%%%%%
%%%%%%%%%%%%%%%%%%%%%%%%%%%%%%%%%%%%%%%

Around the turn of the millennium, \cite{Egami} and \cite{Navas} independently considered the Dirichlet series $\sum_{n=1}^{\infty} F_{n}^{-s}$ for the Fibonacci sequence, establishing its meromorphic continuation to the complex plane. In addition, a problem in the \emph{American Mathematical Monthly} posed this question for a recurrence of degree $2$ of the form $a \alpha^{n} + b \alpha^{-n}$, where $a,b > 0$ and $\alpha > 1$, corresponding to the quadratic polynomial $x^{2} - (\alpha + \alpha^{-1}) x + 1$. In~\cite{Kamano} the same is done for a general quadratic recurrence.

In addition~\cite{Navas} gives explicit formulas for the residues, for the finite values at negative integers, and also some relations for values at positive integers. 
The latter question specifically has also been studied and generalized, for quadratic recurrences in~\cite{ElsnerShimomuraShiokawa,Nakamura, Komatsu0}, for the Tribonacci sequence in~\cite{Komatsu2}, and for a recurrence of arbitrary degree in~\cite{Komatsu1}.

Recently~\cite{SmajlovicLucasHurwitz} consider a Hurwitz-type zeta function, namely arising from the meromorphic continuation of a series of the form $\sum_{n=1}^{\infty} (a_{n} + x)^{-s}$ for a general Lucas quadratic sequence with real coefficients and the same authors in~\cite{SmajlovicTribonacci} consider a general real cubic recurrence, giving the Tribonacci sequence as an example.

Finally, we should mention~\cite{MeherRout1, MeherRout2}, where the authors consider an analogous multiple-Lucas zeta functions and twists of these by Dirichlet characters.

In this paper, we generalize most of these results, considering a general recurrence sequence $a_{n}$ of arbitrary degree and the Dirichlet series $\sum_{n=1}^{\infty} a_{n}^{-s}$. We will restrict ourselves to recurrences over the integers because we are specially interested in arithmetic properties of the finite values of the analytic continuation, for example their rationality at negative integers. Strictly speaking, this restriction is not necessary for many of our results, in particular for establishing the analytic continuation.

We begin in Section~\ref{sec:recurrence-sequences} recalling the basic results we need on recurrence sequences over the integers. In Section~\ref{sec:dirichlet-intro} we take up the study of the associated Dirichlet series and establish its basic properties. Their meromorphic continuation which defines the recurrence zeta function and the determination of its pole set and residues is taken up in Section~\ref{sec:meromorphic-continuation}. Section~\ref{sec:examples} reviews some of the standard examples. Finally, we conclude in Section~\ref{sec:negative integers} with a study of the values of the zeta function at negative integers, proving their rationality.

%%%%%%%%%%%%%%%%%%%%%%%%%%%%%%%%%%%%%%%%%%%%%%%%%%%%%%%%%%%%%%%%%%%%%%%%%%%%%%%%%%%
\section{General remarks about recurrence sequences}
\label{sec:recurrence-sequences}
%%%%%%%%%%%%%%%%%%%%%%%%%%%%%%%%%%%%%%%%%%%%%%%%%%%%%%%%%%%%%%%%%%%%%%%%%%%%%%%%%%%
%%%%%%%%%%%%%%%%%%%%%%%%%%%%%%%%%%%%%%%

Before going into the study of the Dirichlet series associated to a linear recurrence sequence, let us recall a few known facts about such sequences that will be useful in what follows.

Let us formalize what we mean by ``linear recurrence sequence''. First, for a commutative ring $R$ with identity $1\in R$, let $R^{\infty}$ be the set of sequences of elements of $R$, that is, of mappings $N\rightarrow R$. Note that this is an $R$-module. We can define in $R^{\infty}$ the $R$-linear operator $\nabla$ as
\[
	\begin{array}{rccl}
	\nabla: & R^{\infty} & \longrightarrow & R^{\infty} \\
	 & \{a_n\}_{n\in\bN} & \longmapsto & \{a_{n+1}\}_{n\in\bN}\end{array}.
\]

For brevity, we write $\nabla a_n=a_{n+1}$. Since this operator is $R$-linear, it makes sense to consider, for a polynomial $Q(x)\in R[x]$, the operator $Q(\nabla)$, and this way we can give $R^{\infty}$ a $R[x]$-module structure, where
\[
	Q(x)\cdot \{a_n\}=Q(\nabla)\{a_n\}.
\]

This is valid for any ring, but let us specialize now to the case $R=\bZ$, which will be the case in everything that follows.

\begin{definition}\label{d:LRS}
We say that $\{a_n\}\in \bZ^{\infty}$ is a \textbf{linear recurrence sequence} (LRS) if there is a \emph{monic} polynomial $Q(x)\in \bZ[x]$ such that $Q(\nabla)a_n=0$.
\end{definition}

For a LRS $\{a_n\}$ over $\bZ$, we can define its annihilator ideal by
\[
	A_{\bZ}(\{a_n\})=\{Q(x)\in\bZ[x]\:\vert\: Q(\nabla)a_n=0\}\subseteq \bZ[x].
\]

If $\{a_n\}\neq 0$, this is a proper ideal in $\bZ[x]$. Thanks to Gauss' lemma, we have the following result:

\begin{proposition}\label{p:minimal-polynomial}
Let $\{a_n\}$ be a LRS over $\bZ$. There is a unique monic polynomial $P(x)\in\bZ[x]$ such that $A_{\bZ}(\{a_n\})=(P)$.
\end{proposition}

\begin{proof}
Let $Q(x)\in\bZ[x]$ be a monic polynomial in $A_{\bZ}(\{a_n\})$, which exists because $\{a_n\}$ is a LRS over $\bZ$. Since $A_{\bZ}(\{a_n\})\subseteq A_{\bQ}(\{a_n\})$ and $A_{\bQ}(\{a_n\})=(P)$ for some unique monic $P(x)\in\bQ[x]$ (because $\bQ[x]$ is a principal ideal domain), $P(x)$ divides $Q(x)$ and it is the product of some of the irreducible factors of $Q(x)$ over $\bQ$. However, due to Gauss' lemma, these factors are in $\bZ$, and therefore $P(x)\in\bZ[x]$. It is now straightforward that $A_{\bZ}(\{a_n\})=(P)$.
\end{proof}

Thanks to Proposition \ref{p:minimal-polynomial}, the following definition makes sense:

\begin{definition}\label{d:minimal-polynomial}
Let $\{a_n\}$ be a LRS over $\bZ$. The unique monic polynomial $P(x)\in\bZ[x]$ such that $M_{\bZ}(\{a_n\})=(P)$ is called the \textbf{minimal polynomial} of the LRS  $\{a_n\}$. It is minimal in the sense that it is the monic polynomial of least degree in $M_{\bZ}(\{a_n\})$.
\end{definition}

\begin{remark}\label{r:irreducibility-minimal}
If $\{a_n\}$ is annihilated by an irreducible monic polynomial, it is clear from the proof of Proposition \ref{p:minimal-polynomial} that it must be its minimal polynomial. However, unlike the case of algebraic integers, minimal polynomials of LRS need not be irreducible: take, for example, $a_n=2^n+3^n$. It is in a LRS because $a_{n+2}=5a_{n+1}-6a_n$ for all $n\in\bN$. This means that the polynomial $P(x)=x^2-5x+6=(x-2)(x-3)$ annihilates it, but it is easy to check that none of its factors do, so $P(x)$ is the minimal polynomial.
\end{remark}

The last result that we need with respect to LRS  is what we may call a ``Binet-like'' formula, due to its similarity with the Binet formula for the Fibonacci sequence, which this generalises. It can be found in \cite{Everest}, \S\textbf{1.1.6} for any base field with characteristic $0$, we state it for $\bQ$:

\begin{proposition}\label{p:Binet-formula}
Let $\{a_n\}$ be a LRS  over $\bZ$ with minimal polynomial $P(x)\in\bZ[x]$. Let $K$ be the splitting field of $P(x)$ over $\bQ$ and $\alpha_1,...,\alpha_r$ the roots of $P(x)$ in $K$, each with multiplicity $m_i$. There are polynomials $\lambda_1(x),...,\lambda_r(x)\in K[x]$ (in fact, $\lambda_i(x)\in\bQ(\alpha_i)[x]$), with $\deg p_i\leq m_i-1$, such that
\[
	a_n=\lambda_1(n)\alpha_1^n+...+\lambda_r(n)\alpha_r^n \quad \forall n\in\bN.
\]
\end{proposition}

In particular, if the minimal polynomial $P(x)$ of $\{a_n\}$ is separable, $a_n$ can be expressed as a linear combination of the powers of the roots of $P(x)$. In this case, the coefficients $\lambda_i(x)=\lambda_i\in K$ can be obtained easily through a system of linear equations:

\begin{equation}\label{e:coefficient-matrix}
	\left(\begin{array}{c}
	\lambda_1 \\
	\vdots \\
	\lambda_r\end{array}\right)=\left(\begin{array}{ccc}
	\alpha_1 & \hdots & \alpha_r \\
	\vdots & \ddots & \vdots \\
	\alpha_1^r & \hdots & \alpha_r^r\end{array}\right)^{-1}\left(\begin{array}{c}
	a_1 \\
	\vdots \\
	a_r\end{array}\right).
\end{equation}

%%%%%%%%%%%%%%%%%%%%%%%%%%%%%%%%%%%%%%%%%%%%%%%%%%%%%%%%%%%%%%%%%%%%%%%%%%%%%%%%%%%
\section{Dirichlet series defined by a LRS }
\label{sec:dirichlet-intro}
%%%%%%%%%%%%%%%%%%%%%%%%%%%%%%%%%%%%%%%%%%%%%%%%%%%%%%%%%%%%%%%%%%%%%%%%%%%%%%%%%%%
%%%%%%%%%%%%%%%%%%%%%%%%%%%%%%%%%%%%%%%

A general Dirichlet series is an expression of the form
\begin{equation}\label{e:general-dirichlet}
	\sum_{n=1}^{\infty} b_ne^{-\lambda_ns},
\end{equation}
where $s$ is a complex variable, $\{b_n\}$ a sequence of complex numbers and $\lambda_n$ a strictly increasing sequence of nonnegative real numbers such that $\lambda_n\rightarrow\infty$. A classical Dirichlet series is a Dirichlet series for which $\lambda_n=\log n$.

Dirichlet series have been studied in depth and for a long time. Let us give a summary of some well-known results about them here:

There is a real number (possibly $-\infty$ or $+\infty$) $\sigma_c$ such that the series \eqref{e:general-dirichlet} converges for $\sigma=\Re(s)>\sigma_c$ and diverges for $\sigma<\sigma_c$: it is called the convergence abscissa of the series. It can be computed by:
\begin{equation}\label{e:abscissa-convergence}
	\sigma_c=\begin{cases}
		\lim\sup\frac{\log|b_1+...+b_n|}{\lambda_n}\qquad &\text{if }\sum b_k\text{ diverges}, \\
		\lim\sup\frac{\log|b_{n+1}+b_{n+2}+...|}{\lambda_n}\qquad &\text{if }\sum b_k\text{ converges}.\end{cases}
\end{equation}

There is also a real number $\sigma_a$ (again, possible $\pm\infty$) such that the series \eqref{e:general-dirichlet} converges absolutely for $\sigma>\sigma_a$ and does not for $\sigma<\sigma_a$. Applying the formula \eqref{e:abscissa-convergence} to the series $\sum |b_n e^{-\lambda_n s}|$, we get that
\begin{equation}\label{e:abscissa-absolut-convergence}
	\sigma_a=\begin{cases}
	\lim\sup\frac{\log(|b_1|+...+|b_n|)}{\lambda_n}\qquad &\text{if }\sum |b_k|\text{ diverges}, \\
	\lim\sup\frac{\log(|b_{n+1}|+...)}{\lambda_n}\qquad &\text{if }\sum|b_k|\text{ converges}.\end{cases}
\end{equation}

It is always true that
\[
	0\leq\sigma_a-\sigma_c\leq\lim\sup\frac{\log n}{\lambda_n}.
\]

For the proof of all these statements, the reader can check \S 207 in \cite{Landau}.

In the case of classical Dirichlet series, $0\leq\sigma_a-\sigma_c\leq 1$ (for example, the alternated Dirichlet series $\zeta(s)=\sum (-1)^{n+1}n^{-s}$ has $\sigma_c=0$ and $\sigma_a=1$).

We want to study Dirichlet series of the type
\[
	\sum_{n=1}^{\infty} \frac{1}{a_n^s}=\sum_{n=1}^{\infty} e^{-s\log a_n },
\]
where $\{a_n\}$ is a LRS  over $\bZ$. Note that, if this is to be coherent with our definition of a Dirichlet series, we need the sequence $\{a_n\}$ to be a strictly increasing sequence of positive integers. However, we can take a finite number of terms out of the sum \eqref{e:general-dirichlet} without changing its nature too much (since a finite sum of exponential functions is a holomorphic function), we will allow for sequences that are \emph{eventually} strictly increasing, that is, sequences for which there is some $n_0\in\bN$ such that $\{a_{n+n_0}\}_{n\in\bN}$ is strictly increasing (eventual positivity follows from this). A sufficient condition for this to happen is the following:

\begin{proposition}\label{p:esi-sequence}
Let $\{a_n\}$ be a LRS over $\bZ$ with minimal polynomial $P(x)$. Let $\alpha_1,...,\alpha_r$ be the roots of $P(x)$, each with multiplicity $m_i$, and put
\[
	a_n=\lambda_1(n)\alpha_1^n+...+\lambda_r(n)\alpha_r^n,
\]
as in Proposition \ref{p:Binet-formula}. If there is one root $\alpha_{i_0}$, with $\lambda_{i_0}(n)$ not identically $0$, verifying $|\alpha_{i_0}|>|\alpha_i|$ for every $i\neq i_0$ such that $\lambda_{i_0}(n)$ is not identically $0$ and $|\alpha_{i_0}|>1$. There are two exhaustive and mutually exclusive possibilities:
\begin{enumerate}[wide, labelwidth=!, labelindent=0pt, label=\arabic*.]
\item Either $\{a_n\}$ or $\{-a_n\}$ is eventually strictly increasing.

\item $\{a_n\}$ has an infinite number of sign changes as $n$ grows to $\infty$.
\end{enumerate}
\end{proposition}

\begin{proof}
Assume, without loss of generality, that $i_0=1$.

First, note that the hypothesis implies that $\alpha_1\in\mathbb{R}$: $A_{\bC}(\{a_n\})$ contains the polynomial $Q(x)=(x-\alpha_{i_1})...(x-\alpha_{i_t})$, where $\{i_1,...,i_t\}$ is the set of indices such that $\lambda_i(n)$ is not identically $0$. If it were not $Q(x)\in\bZ[x]$, then $\{a_n\}$ could not be a sequence of integers. This means that, if $\lambda_{j_1}(n)$ is not identically $0$ and $\alpha_{j_2}$ is the conjugate root of $\alpha_{j_1}$, then $\lambda_{j_2}(n)$ is also not identically $0$. Since, by hypothesis, there is no other root with modulus as large as $|\alpha_1|$, $\alpha_1$ does not have a complex conjugate and it is real.

Now, since $\lambda_1(x)$ is a polynomial (and it must have real coefficients since $\alpha_1\in\bR$), there is some $n_0\gg0$ such that $\lambda_1(n)$ has constant sign, equal to the sign of the leading coefficient of $\lambda_1(x)$, for every $n\geq n_0$. In particular, its nonzero, so for $n\geq n_0$ we can put
\[
	a_n=\lambda_1(n)\alpha_1^n\left(1+\frac{\lambda_2(n)}{\lambda_1(n)}\left(\frac{\alpha_2}{\alpha_1}\right)^n+...+\frac{\lambda_r(n)}{\lambda_1(n)}\left(\frac{\alpha_r}{\alpha_1}\right)^n\right).
\]
For every $i=2,...,r$, it is clear that
\[
	\lim_{n\rightarrow\infty}\frac{\lambda_i(n)}{\lambda_1(n)}\left(\frac{\alpha_r}{\alpha_1}\right)^n=0,
\]
so there is some $n_1>n_0$ such that, for $n\geq n_1$, $a_n$ has the same sign as $\lambda_1(n)\alpha_1^n$. If $\alpha_1>1$, this sign is that of the leading coefficient of $\lambda_1(x)$, and the fact that $a_n$ is eventually strictly increasing is obvious since $\alpha_1>1$. If $\alpha_1<-1$, this sign changes with the parity of $n$.
\end{proof}

\begin{remark}\label{r:other-esi-sequences}
The condition used in Proposition \ref{p:esi-sequence} is, while sufficient, not always necessary for a LRS  to be eventually strictly increasing sequence. For example, the polynomial $x^2-2$ has 2 roots with the same absolute value, and for any initial terms $a_1,a_2$ such that $0<a_1<a_2<2a_1$, the LRS  $\{a_n\}$ that this defines is strictly increasing. However, the existence of one dominant root $>1$ will be a necessary hypothesis in the proof of Theorem \ref{t:meromorphic continuation}, which is our main result, so there is no need to study here other kinds of eventually strictly increasing LRS .
\end{remark}

From now on we will deal with LRS  $\{a_n\}$ whose minimal polynomial is irreducible. From the discussion in the proof of Proposition \ref{p:esi-sequence} it is clear that, in this case, all the coefficients $\lambda_i(n)$ are not identically $0$. Furthermore, since irreducible polynomials are also separable, these coefficients are constant. Then, if there is a positive \emph{dominant root} (a root whose modulus is greater than that of the other roots), $\{a_n\}$ is automatically eventually strictly increasing or eventually strictly decreasing, because this root (except on the case $P(x)=(x-1)$, which gives constant constant LRS ) is also automatically greater than $1$.

Let $\{a_n\}$ be a LRS over $\bZ$ such that its minimal polynomial $P(x)$ is irreducible and has a positive dominant root ($\alpha_1$) and its coefficient in the Binet formula for $a_n$ is positive. Since there is some $n_0\gg0$ such that $\{a_{n+n_0}\}_{n\in\bN}$ is strictly increasing and positive, we will assume, without loss of generality in what follows, that $\{a_n\}$ is strictly increasing and positive. We can therefore define the Dirichlet series associated to the sequence $\{a_n\}$ as
\[
	\sum_{n=1}^{\infty}\frac{1}{a_n^s}=\sum_{n=1}^{\infty} e^{-\log s a_n }
\]

Using both formula \eqref{e:abscissa-convergence} and Binet's formula \ref{p:Binet-formula} we can deduce easily that its abscissa of convergence is
\[
	\sigma_c=\lim\sup \frac{\log n}{\log a_n}=\lim \frac{\log n}{\log \lambda_1+n\log\alpha_1}=0,
\]
and since the coefficients $b_n=1$ of the series are all positive, the same goes for its abscissa of absolute convergence. In the next section we will show that this series, that defines a holomorphic function in the half-plane $\sigma>0$, has a meromorphic continuation to the whole $s$-plane, and we will find its poles and residues.

%%%%%%%%%%%%%%%%%%%%%%%%%%%%%%%%%%%%%%%%%%%%%%%%%%%%%%%%%%%%%%%%%%%%%%%%%%%%%%%%%%%
\section{Meromorphic continuation of the Dirichlet series}
\label{sec:meromorphic-continuation}
%%%%%%%%%%%%%%%%%%%%%%%%%%%%%%%%%%%%%%%%%%%%%%%%%%%%%%%%%%%%%%%%%%%%%%%%%%%%%%%%%%%
%%%%%%%%%%%%%%%%%%%%%%%%%%%%%%%%%%%%%%%

Let, as before, $\{a_n\}$ be a LRS  of integer numbers with minimal polynomial $P(x)\in\bZ[x]$ irreducible and with a positive dominant root. Let $\alpha_1,...,\alpha_r$ be the roots of $P(x)$ (ordered by decreasing modulus, so that $\alpha_1$ is the dominant root) and $\lambda_1,...,\lambda_r$ nonzero numbers in the splitting field of $P(x)$ such that
\[
	a_n=\lambda_1\alpha_1^n+...+\lambda_r\alpha_r^n.
\]

Assume once again, without loss of generality, that $\{a_n\}$ is strictly increasing and positive (that is, $\lambda_1>0$). Then, as per the discussion at the end of the previous section, the Dirichlet series
\[
	\sum_{n=1}^{\infty}\frac{1}{a_n^s}
\]
converges and does so absolutely in the half-plane $\sigma=\Re(s)>0$, so it defines a holomorphic function $\varphi(s)$ there.

\begin{theorem}\label{t:meromorphic continuation}
With the previous hypotheses, the holomorphic function $\varphi(s)$ defined by the Dirichlet series associated to $\{a_n\}$ can be continued analytically to a meromorphic function on $\bC$ (that we will still call $\varphi(s)$) whose only singularities are simple points at the poles
\begin{equation}\label{e:poles}
 s_{n,k_1,...,k_{r-1}}=\frac{\log|\alpha_1^{-k_1}\alpha_2^{k_1-k_2}...\alpha_r^{k_{r-1}}|}{\log\alpha_1}+i\frac{\arg(\alpha_1^{-k_1}\alpha_2^{k_1-k_2}...\alpha_r^{k_{r-1}})+2\pi n}{\log\alpha_1},
\end{equation}
where the parameters $n,k_1,...,k_{r-1}$ are integer numbers satisfying
\[
	n\in\bZ, \qquad k_1\geq 0, \qquad 0\leq k_i\leq k_{i-1}\quad i=2,...,r-1.
\]
\end{theorem}

\begin{proof}
Using Binet's formula for $a_n$ we can expand $a_n^{-s}$ using the binomial series to obtain
\[
	a_n^{-s}=\sum_{k_1=0}^{\infty}\sum_{k_2=0}^{k_1}...\sum_{k_{r-1}=0}^{k_{r-2}}\binom{-s}{k_1}\binom{k_1}{k_2}...\binom{k_{r-2}}{k_{r-1}}(\lambda_1\alpha_1^n)^{-s-k_1}(\lambda_2\alpha_2^n)^{k_1-k_2}...(\lambda_r\alpha_r^n)^{k_{r-1}}.
\]

Note that this is technically only valid for the $a_n$ in which
\[
	\left|\frac{\lambda_2\alpha_2^n+...+\lambda_r\alpha_r^n}{\lambda_1\alpha_1^n}\right|<1,
\]
but, due to the assumption that $\alpha_1$ is a dominant root, this holds for all but a finite number of $n$, and we can assume without loss of generality for the purpose of analytic continuation that it is the case for every $n\in\bN$, since the terms of the Dirichlet series where it does not hold add up to an entire function.

In order to try to simplify our notation a bit we set
\[
	k_{\leq j}=\{k_1,...,k_j\},
\]
so that
\[
	\sum_{k_1=0}^{\infty}\sum_{k_2=0}^{k_1}...\sum_{k_{r-1}=0}^{k_{r-2}}=\sum_{k_{\leq r-1}=0}^{\infty, k_{\leq r-2}}\quad \text{ and }\quad \binom{-s}{k_1}\binom{k_1}{k_2}...\binom{k_{r-2}}{k_{r-1}}=\binom{-s,k_{\leq r-2}}{k_{\leq r-1}}.
\]

This way, the Dirichlet series $\sum a_n^{-s}$ can be written as
\[
	\sum_{n=1}^{\infty}\sum_{k_{\leq r-1}=0}^{\infty, k_{\leq r-2}}\binom{-s,k_{\leq r-2}}{k_{\leq r-1}}(\lambda_1\alpha_1^n)^{-s-k_1}(\lambda_2\alpha_2^n)^{k_1-k_2}...(\lambda_r\alpha_r^n)^{k_{r-1}},
\]
or
\begin{equation}\label{e:dir-before}
	\sum_{n=1}^{\infty}\sum_{k_{\leq r-1}=0}^{\infty,k_{\leq r-2}}\Lambda_{k_{\leq r-1}}(s)\left(\alpha_1^{-s-k_1}\alpha_2^{k_1-k_2}...\alpha_r^{k_{r-1}}\right)^n,
\end{equation}
where we write
\[
	\Lambda_{k_{\leq r-1}}(s)=\binom{-s,k_{\leq r-2}}{k_{\leq r-1}}\lambda_1^{-s-k_1}\lambda_2^{k_1-k_2}...\lambda_r^{k_{r-1}},
\]
to shorten the expression further. Since this functions have no poles, they do not change the reasoning that follows.

Now let us see that the series \eqref{e:dir-before} is absolutely convergent as a \emph{double} series for $\sigma=\Re(s)>0$. For the functions $\Lambda_{k_{\leq r-1}}(s)$ we have the estimate, as in \cite{Navas},
\[
	\left|\Lambda_{k_{\leq r-1}}\right|\leq (-1)^{k_1}\binom{-|s|,k_{\leq r-2}}{k_{\leq r-1}}\lambda_1^{-\sigma-k_1}|\lambda_2|^{k_1-k_2}...|\lambda_r|^{k_{r-1}}.
\]

Using this, we get that
\begin{equation}\label{e:longchain}
\begin{aligned}
	&\sum_{n=1}^{\infty}\sum_{k_{\leq r-1}=0}^{\infty, k_{\leq r-2}}\left|\Lambda_{k_{\leq r-1}}(s)\left(\alpha_1^{-s-k_1}\alpha_2^{k_1-k_2}...\alpha_r^{k_{r-1}}\right)^n\right|\leq \\
	\leq &\sum_{n=1}^{\infty}\sum_{k_{\leq r-1}=0}^{\infty, k_{\leq r-2}}(-1)^{k_1}\binom{-|s|,k_{\leq r-2}}{k_{\leq r-1}}(\lambda_1\alpha_1^n)^{-\sigma-k_1}...|\lambda_n\alpha_r^n|^{k_{r-1}}= \\
	=&\lambda_1^{-\sigma}\sum_{n=1}^{\infty}\alpha_1^{-n\sigma}\left(1-\frac{|\lambda_2\alpha_2^n|+...+|\lambda_r\alpha_r^n|}{\lambda_1\alpha_1^n}\right)^{-|s|}\leq \\
	\leq &\lambda_1^{-\sigma}\left(1-\frac{|\lambda_2\alpha_2|+...+|\lambda_r\alpha_r|}{\lambda_1\alpha_1}\right)^{-|s|}\sum_{n=1}^{\infty}\alpha_1^{-n\sigma}<+\infty.
\end{aligned}
\end{equation}

Note that the step between lines 3 and 4 assumes that the sequence of real numbers
\[
	\frac{|\lambda_2\alpha_2^n|+...+|\lambda_r\alpha_r^n|}{\lambda_1\alpha_1}
\]
is strictly decreasing and always less than $1$. Since this is eventually true (that is, true for $n\geq n_0$ for some $n_0\in\bN$) and, again, a finite number of the terms of the series only change the total by an entire function, we can assume $n_0=0$ without loss of generality.

From the reasoning \eqref{e:longchain}, the double series \eqref{e:dir-before} is absolutely convergent for $\sigma>0$, so we can change the order of summation and, since
\[
	\left|\alpha_1^{-s-k_1}\alpha_2^{k_1-k_2}...\alpha_r^{k_{r-1}}\right|=\alpha_1^{-\sigma-k_1}|\alpha_2|^{k_1-k_2}...|\alpha_r|^{k_{r-1}}<1
\]
for $\sigma>0$, $k_1\geq 0$, the double series becomes
\begin{equation}\label{e:dir-after}
	\varphi(s)=\sum_{k_{\leq r-1}=0}^{\infty, k_{\leq r-2}} \Lambda_{k_{\leq r-1}}(s)\frac{\alpha_1^{-s-k_1}\alpha_2^{k_1-k_2}...\alpha_r^{k_{r-1}}}{1-\alpha_1^{-s-k_1}\alpha_2^{k_1-k_2}...\alpha_r^{k_{r-1}}}.
\end{equation}

It remains to see that $\varphi(s)$ is a meromorphic function. Once this is done, the statement about its poles is trivial. To do this, we claim that the series \eqref{e:dir-after} converges uniformly on every compact set of the $s$-plane that does not contain any of the points $s_{n,k\leq r-1}$ defined as in \eqref{e:poles}.

Let us set
\[
	f_{k_{\leq r-1}}(s)=\Lambda_{k_{\leq r-1}}(s)\frac{\alpha_1^{-s-k_1}\alpha_2^{k_1-k_2}...\alpha_r^{k_{r-1}}}{1-\alpha_1^{-s-k_1}\alpha_2^{k_1-k_2}...\alpha_r^{k_{r-1}}}.
\]
For any $s\in\bC$ we have, for $k_1\gg0$,
\[
\begin{aligned}
	\left|\alpha_1^{s+k}\alpha_2^{-k_1+k_2}...\alpha_r^{-k_{r-1}}-1\right|&\geq \alpha_1^{\sigma+k_1}|\alpha_2|^{-k_1+k_2}...|\alpha_r|^{-k_{r-1}}-1> \\
	&>\alpha_1^{\sigma+k_1-1}|\alpha_2|^{-k_1+k_2}...|\alpha_r|^{-k_{r-1}},
\end{aligned}
\]
so there is some $k_0\gg0$ such that
\[
\begin{aligned}
	&\sum_{k_1=k_0}^{\infty}\sum_{k_2=0}^{k_1}...\sum_{k_{r-1}=0}^{k_{r-2}}\left|f_{k_{\leq r-1}}(s)\right|\leq \\
	&\leq \lambda_1^{-\sigma}\alpha_1^{-\sigma+1}\sum_{k_{\leq r-1}=0}^{\infty, k_{\leq r-2}} (-1)^{k_1}\binom{-|s|,k_{\leq r-2}}{k_{\leq r-1}}(\lambda_1\alpha_1)^{-k_1}...|\lambda_r\alpha_r|^{k_{r-1}}= \\
	&=\lambda_1^{-\sigma}\alpha_1^{-\sigma+1}\left(1-\frac{|\lambda_2\alpha_2|+...+|\lambda_r\alpha_r|}{\lambda_1\alpha_1}\right)^{-|s|}<+\infty.
\end{aligned}
\]
This bound is uniform when $s$ varies in a compact set, so we get that $\varphi(s)$, as defined by \eqref{e:dir-after}, is a meromorphic function of $s$ with simple poles at the points stated in \eqref{e:poles}.
\end{proof}

\begin{remark}\label{r:several-poles}
Note that we have not said that all the points $s_{n,k_{\leq r-1}}$ are distinct. It could happen that, for two different $r$-tuples $n,k_1,...,k_{r-1}$ and $n',k_1',...,k_{r-1}'$, $s_{n,k_{\leq r-1}}=s_{n',k'_{\leq r-1}}$. However, it is not hard to check that any pole $s_0$ of $\varphi(s)$ can only be reached by a \emph{finite} number of $r$-tuples, so this observation does not change the truth of the statements in Theorem \ref{t:meromorphic continuation}.

Though in general we do not know the answer to this question (i.e., \emph{Are the points $s_{n,k_{\leq r-1}}$ distinct for different $r$-tuples?}), it is possible to answer it in the affirmative in particular cases, as we will see later in some examples.
\end{remark}

\begin{corollary}\label{c:residues}
Following the notation of Theorem \ref{t:meromorphic continuation}, if $s_0\in\bC$ is a pole of $\varphi(s)$ and $\bm\tau(s_0)$ denotes the set of $r$-tuples $(n,k_1,...,k_{r-1})$ such that $s_{n,k_{\leq r-1}}=s_0$, the residue of $\varphi(s)$ at the pole $s=s_0$ is
\begin{equation}\label{e:residues}
\sum_{(n,k_{\leq r-1})\in\bm\tau(s_0)}\frac{1}{\log\alpha_1}\lambda_1^{-s_0-k_1}\lambda_2^{k_1-k_2}..\lambda_r{k_{r-1}}\binom{-s_0}{k_1}\binom{k_1}{k_2}...\binom{k_{r-2}}{k_{r-1}}.
\end{equation}
\end{corollary}

\begin{proof}
First of all, note that, as we stated in Remark \ref{r:several-poles}, if $s_0\in\bC$ is a pole of $\varphi(s)$ then $\bm\tau(s_0)$ is nonempty and finite. Therefore, there is no obstacle to the calculation
\[
\begin{aligned}
	\res_{s=s_0}\varphi(s)&=\lim_{s\rightarrow s_0}(s-s_0)\varphi(s)= \\
	&=\sum_{\bm\tau(s_0)}\lim_{s\rightarrow s_0}(s-s_0)f_{k_{\leq r-1}}(s_0)=\\
	&=\sum_{\bm\tau(s_0)}\lim_{s\rightarrow s_0}\frac{(s-s_0)\Lambda_{k_{\leq r-1}}(s)\alpha_1^{-s-k_1}\alpha_2^{k_1-k_2}...\alpha_r^{k_{r-1}}}{1-\alpha_1^{-s-k_1}\alpha_2^{k_1-k_2}...\alpha_r^{k_{r-1}}}= \\
	&=\sum_{\bm\tau(s_0)}\frac{\Lambda_{k_{\leq r-1}}(s_0)}{\log\alpha_1},
\end{aligned}
\]
and using the explicit definition of $\Lambda_{k_{\leq r-1}}(s)$ we get \eqref{e:residues}.
\end{proof}

%%%%%%%%%%%%%%%%%%%%%%%%%%%%%%%%%%%%%%%%%%%%%%%%%%%%%%%%%%%%%%%%%%%%%%%%%%%%%%%%%%%
\section{Some particular series}
\label{sec:examples}
%%%%%%%%%%%%%%%%%%%%%%%%%%%%%%%%%%%%%%%%%%%%%%%%%%%%%%%%%%%%%%%%%%%%%%%%%%%%%%%%%%%
%%%%%%%%%%%%%%%%%%%%%%%%%%%%%%%%%%%%%%%

Let us now study some concrete examples of Dirichlet series associated to a LRS .

\begin{example}[Quadratic recurrence]
\label{ex:quadratic-sequence}
The first case that we focus on, which is both the simplest and the most studied, is the quadratic recurrence relation. That is, the LRS  $\{a_n\}$ that defines our Dirichlet series has an irreducible minimal polynomial $P(x)\in\bZ[x]$ of degree $2$ with a dominant root.

The most famous of this sequence is the Fibonacci sequence $\{F_n\}$, with initial terms $F_1=F_2=1$ and minimal polynomial $x^2-x-1$. Its associated Dirichlet series was studied in depth by Navas in \cite{Navas} and by Egami in \cite{Egami} (although note that the list of poles given there also contains some points which are in fact removable singularities). Its companion, the Lucas sequence, and their generalizations to LRS  of order $2$ also define Dirichlet series that were studied by Kamano in \cite{Kamano}. There are also studies about the values of these series on positive integers, like \cite{Komatsu0} or \cite{Nakamura}. The meromorphic continuation for a general sequence of order $2$ with positive coefficients is studied also in \cite{Silverman}. More recently~\cite{SmajlovicLucasHurwitz} generalizes these series to a Hurwitz-type zeta function associated to a Lucas sequence.

If the minimal polynomial $P(x)=x^2+p_1x+p_0\in\bZ[x]$ is irreducible and has one dominant root $\alpha>1$, the other root will be $\tilde{\alpha}=p_0\alpha^{-1}$, or $\norm(\alpha)\alpha^{-1}$, if $\norm$ denotes the absolute norm in $\bQ(\alpha)$. Then, our formula \eqref{e:poles} for the poles of its associated Dirichlet series yields
\[
	s_{n,k}=-k\left(2-\frac{\log |\norm(\alpha)|}{\log\alpha}\right)+i\frac{k\arg(\tilde{\alpha})+2\pi n}{\log\alpha_1}, \qquad n\in\bZ, \quad k\geq 0.
\]

Note that $\arg(\tilde{\alpha})$ is either $0$ or $\pi$, so in any case the imaginary part of $s_{n,k}$ is an integer multiple of $\pi$ divided by $\log\alpha_1$.

It is clear that in this case, all of these points are distinct (remember that, as per Remark \ref{r:several-poles}, this is not proved in the general case), and they form a semi-lattice in the half-plane $\Re(s)\leq 0$. If, for example, $\norm(\alpha)=\pm 1$, that is, $\alpha$ is a unit (as is the case, for example, in the Fibonacci sequence), these poles will be arranged in vertical lines with abscissa $-2k$, $k\geq 0$, and the other direction of the semi-lattice will be horizontal or oblique depending on whether $\norm(\alpha)=+1$ or $\norm(\alpha)=-1$.

Navas proved, in \cite{Navas}, that, in the case of the Fibonacci Dirichlet series, the values at the negative integers that are not poles are rational numbers, and derived formulas for them that depend on the Fibonacci and Lucas numbers. The first part of this result is proved in general in the following section \S\ref{sec:negative integers}.
\end{example}

\begin{example}[Cubic recurrence]
\label{ex:cubic-sequences}
The second case, which has some more variety and is much less known, is the cubic recurrence relation. That is, the LRS  $\{a_n\}$ that defines our Dirichlet series has an irreducible minimal polynomial $P(x)\in\bZ[x]$ of degree $3$ with a dominant root. An example would be the \emph{Tribonacci} sequence, defined by the recurrence relation $T_{n+3}=T_{n+2}+T_{n+1}+T_n$ with initial values $T_1=T_2=1$, $T_3=2$ (a displacement of sequence \href{https://oeis.org/A000073}{A000073} in \cite{OEIS} or \cite{Koshy}). The values at positive integers of the corresponding Dirichlet series were studied in \cite{Komatsu2}, while the function defined by the series has been recently examined in~\cite{SmajlovicTribonacci}, considering arbitrary real coefficients.

In general, for an irreducible polynomial $P(x)\in\bZ[x]$ with one dominant root $>1$, we have two possibilities: either the other two roots are complex conjugate, or they are both real (this last possibility, characterized by the discriminant of the polynomial being positive, is the \emph{casus irreducibilis}).

In the case were there are two complex conjugate roots, we have one real root $\alpha_1>1$ and two complex roots $\alpha_2$, $\alpha_3=\overline{\alpha_2}$. We have, ordering the roots appropriately $\alpha_2=\rho e^{i\theta}$, with $0<\rho<\alpha_1$ and $0<\theta<\pi$, and $\alpha_3=\rho e^{-i\theta}$. Once again, we can obtain $\norm(\alpha_1)=\alpha_1\rho^2\in\bN$, and the poles of the continuation of the Dirichlet series are the points
\[
	s_{n,k,l}=-\frac{k}{2}\left(3-\frac{\log \norm(\alpha_1)}{\log\alpha_1}\right)+i\frac{(k-2l)\theta+2\pi n}{\log\alpha_1}, \quad n\in\bZ, \quad 0\leq l\leq k.
\]

These poles, unlike in the case of quadratic recurrences, do not form a semi-lattice: they are still grouped by vertical lines (if, for example, $\alpha_1$ is a unit, they are vertical lines of abscissa $-\frac{3}{2}k$, $k\geq 0$), but in each of these lines the poles are increasingly ``blurry'' when we get away from the real axis.

Since the real part of these poles depends only on one parameter, whether or not all of these points are distinct depends solely on whether, for fixed $k$, $(k-2l)\theta+2\pi n$ can take repeated values when $0\leq l\leq k$ and $n\in\bZ$. This is equivalent to asking whether or not $\theta$ is a rational multiple of $\pi$, which is equivalent to asking whether $e^{i\theta}=\frac{\alpha_2}{|\alpha_2|}$ is a root of unity.

This question has a negative answer, for example, in the case of the Tribonacci sequence: $\alpha_1$, $\alpha_2$ and $\alpha_3$ are the roots of $x^3-x^2-x-1$, and for these values, $e^{i\theta}$ has minimal polynomial $x^{12} +  4x^{10} + 11 x^8+12x^6+11x^4+4x^2+1$, which does not divide $x^n-1$ for any $n$ because it has roots with modulus different from $1$. Thus, the poles $s_{n,k,l}$ of the Tribonacci Dirichlet series are all distinct. We do not know, in general, if this is the case for any cubic recurrence sequence with two complex roots.

For the \emph{casus irreducibilis}, that is, the case where $\alpha_1>1$ and $\alpha_2,\alpha_3\in\bR$, the structure of the poles is a bit different. Formula \eqref{e:poles} gives us the points
\[
	s_{n,k,l}=\frac{-1}{\log\alpha_1}\left(k\frac{\log|\alpha_2|}{\log\alpha_1}-l\frac{\log|\alpha_2|}{\log|\alpha_3|}\right)+i\frac{k\arg(\alpha_2)+l\arg(\alpha_2^{-1}\alpha_3)+2\pi n}{\log\alpha_1},
\]
where once again $n\in\bZ$, $0\leq l\leq k$.

Since $\alpha_2,\alpha_3\in\bR$, the imaginary part is always a multiple of $\frac{\pi}{\log\alpha_1}$, so these points are grouped into \emph{horizontal half-lines} in the half-plane $\Re(s)\leq 0$, where in each horizontal line the points get ``blurrier'' when we travel away from the imaginary axis.
\end{example}

\begin{example}[$N$-bonacci sequences]
\label{ex:N-bonacci}
Let us speak shortly of one last example of these kind of Dirichlet series, this time associated to a sequence $\{a_n\}$ satisfying the $N$-bonacci recurrence relation, i.e., whose minimal polynomial is $P_\norm(x)=x^N-x^{N-1}-...-x-1$. These polynomials are all irreducible, have one real root $\phi_N$ in the interval $(1,2)$ and all their other roots lie in the open unit disk (see \cite{Miller}). In fact, all other roots are complex if $N$ is odd, and there is exactly one more real root, which is negative, when $N$ is even. Furthermore, they are all irreducible (as a consequence of \emph{Theorem 2.2.5} in \cite{Prasolov}). It can even be proved that $\{\phi_N\}$ is a strictly increasing sequence with limit $2$.

Recurrence sequences satisfying some $N$-bonacci recurrence relation are sometimes called \emph{generalized Fibonacci numbers} or \emph{$N$-bonacci numbers} (we have the Tetranacci numbers, \href{https://oeis.org/A000288}{A000288} in \cite{OEIS}, the Pentanacci numbers, \href{https://oeis.org/A000322}{A000322}, and so forth). Since the $N$-bonacci polynomial satisfies our hypotheses, our results also apply to Dirichlet series associated to $N$-bonacci sequences.
\end{example}

%%%%%%%%%%%%%%%%%%%%%%%%%%%%%%%%%%%%%%%%%%%%%%%%%%%%%%%%%%%%%%%%%%%%%%%%%%%%%%%%%%%
\section{Values at negative integers}
\label{sec:negative integers}
%%%%%%%%%%%%%%%%%%%%%%%%%%%%%%%%%%%%%%%%%%%%%%%%%%%%%%%%%%%%%%%%%%%%%%%%%%%%%%%%%%%
%%%%%%%%%%%%%%%%%%%%%%%%%%%%%%%%%%%%%%%

In \cite{Navas}, Navas proves that the values of the meromorphic continuation of the Fibonacci Dirichlet series at the negative integers that are not poles are rational numbers that can be written in terms of the Fibonacci and Lucas sequences. We ask now what can be said, in the general setting, about the values at negative integers of our meromorphic continuation.

Clearly, some of these negative integers may be poles: for example, the points $-2m$, $m\in\bN$, for a recurrence sequence of degree $2$ whose principal root is a totally positive unit (see Example \ref{ex:quadratic-sequence}), or the points $-3m$, $m\in\bN$, for a recurrence sequence of degree $3$ with two complex roots (see Example \ref{ex:cubic-sequences}). At the points which that is not the case, we can prove that the values of the function are once again rational numbers.

Let us keep the same assumptions as in Section \ref{sec:meromorphic-continuation}, that is, $\{a_n\}$ a LRS  of integers which has an irreducible minimal polynomial $P(x)\in\bZ[x]$ with a positive dominant root, and assume (once again, without loss of generality for our purposes) that $\{a_n\}$ is strictly increasing and positive. If $\alpha_1,...,\alpha_r$ are the roots of $P(x)$ (ordered with decreasing modulus), we have
\[
	a_n=\lambda_1\alpha_1^n+...+\lambda_r\alpha_r^n \quad \forall n\in\bN
\]
for unique $\lambda_1,...,\lambda_r$ in the splitting field of $P(x)$, and the function $\varphi(s)$ defined as in \eqref{e:dir-after} represents a meromorphic continuation of the Dirichlet series $\sum a_n^{-s}$ to the whole $s$-plane.

In the following proof we will use multi-index notation, so we introduce it here. A multi-index of \emph{length} $r$ is an element $\bm\beta=(\beta_1,...,\beta_r)\in\bN_0^r$. Given $\bm\beta$, we write:
\begin{enumerate}[label=$\bullet$]
\item $|\bm\beta|=\beta_1+...+\beta_r$.

\item $\displaystyle\binom{|\bm\beta|}{\bm\beta}=\frac{|\bm\beta|!}{\beta_1!...\beta_r!}$.

\item Given $\bm z=(z_1,...,z_r)\in\bC^r$, ${\bm z}^{\bm\beta}=z_1^{\beta_1}...z_r^{\beta_r}$.
\end{enumerate}

\begin{theorem}\label{t:values-negative-integers}
With our previous hypothesis, if $m\in\bN$ verifies that $s=-m$ is not a pole of $\varphi(s)$, then $\varphi(-m)\in\bQ$.
\end{theorem}

\begin{proof}
For $m\in\bN$, the binomial coefficient
\[
	\binom{-(-m)}{k_1}=\binom{m}{k_1}
\]
is nonzero only if $k_1\leq m$, so the sum in $k_1$ that appears in the expression \eqref{e:dir-after} becomes a finite sum, that is,
\[
	\varphi(-m)=\sum_{k_{\leq r-1=0}}^{m,k_{\leq r-2}}\Lambda_{k_{\leq r-1}}(-m)\frac{\alpha_1^{m-k_1}\alpha_2^{k_1-k_2}...\alpha_r^{k_{r-1}}}{1-\alpha_1^{m-k_1}\alpha_2^{k_1-k_2}...\alpha_r^{k_{r-1}}}.
\]
Since the sum of the exponents of the $\alpha_i$ (and those of the $\lambda_i$ in the full expansion of $\Lambda_{k_{\leq r-1}}(-m)$) equals $m$ in every term, this expression can be written more cleanly using multi-indices,
\begin{equation}\label{e:phi-m}
	\varphi(-m)=\sum_{|\bm\beta|=m}\binom{m}{\bm\beta}{\bm\lambda}^{\bm\beta}\frac{{\bm\alpha}^{\bm\beta}}{1-{\bm\alpha}^{\bm\beta}},
\end{equation}
where $\bm\lambda=(\lambda_1,...,\lambda_r)$ and $\bm\alpha=(\alpha_1,...,\alpha_r)$.

We will show that \eqref{e:phi-m} is an element of $\bQ(\alpha_1,...,\alpha_r)^{\cS_r}$, that is, a rational function of $\alpha_1,...,\alpha_r$ invariant under the action of the symmetric group $\cS_r$. This means in particular that it is an element of the decomposition field of $P(x)$ invariant under its Galois group, so it is a rational number.

We can write
\begin{equation}\label{e:phi-m-fraction}
	\varphi(-m)=\frac{\displaystyle\sum_{\bm\beta|=m}\binom{m}{\bm\beta}(\bm\lambda\bm\alpha)^{\bm\beta}\underset{\bm\beta'\neq\bm\beta}{\prod_{|\bm\beta'|=m}}(1-{\bm\alpha}^{\bm\beta'})}{\displaystyle\prod_{|\bm\beta|=m}(1-{\bm\alpha}^{\bm\beta})}
\end{equation}

The denominator of this expression is evidently invariant under permutations of the $\alpha_i$, so it is an integer. Since $-m$ is not a pole of $\varphi(s)$, it must be nonzero. It remains to see that the numerator in \eqref{e:phi-m-fraction} is a rational number.

The main difficulty here is dealing with the $\lambda_i$. Note, however, that from Binet's formula we can obtain the matrix equation
\[
	\left(\begin{array}{ccc}
	\alpha_1 & \hdots & \alpha_r \\
	\vdots & \ddots & \vdots \\
	\alpha_1^r & \hdots & \alpha_r^r\end{array}\right)
	\left(\begin{array}{c}
	\lambda_1 \\
	\vdots \\
	\lambda_r\end{array}\right)=
	\left(\begin{array}{c}
	a_1 \\
	\vdots \\
	a_r\end{array}\right),
\]
and using Cramer's rule we get that each $\lambda_i$ is
\begin{equation}\label{e:lambda-cramer}
	\lambda_i=
	\left|\begin{array}{ccccccc}
	\alpha_1 & \hdots & \alpha_{i-1} & a_1 & \alpha_{i+1} & \hdots & \alpha_r \\
	\vdots & & \vdots & \vdots & \vdots & & \vdots \\
	\alpha_1^r & \hdots & \alpha_{i-1}^r & a_r & \alpha_{i+1}^r & \hdots & \alpha_r^r\end{array}\right|
	\left|\begin{array}{ccc}
	\alpha_1 & \hdots & \alpha_r \\
	\vdots & \ddots & \vdots \\
	\alpha_1^r & \hdots & \alpha_r^r\end{array}\right|^{-1}.
\end{equation}

The second of the determinants in \eqref{e:lambda-cramer} is clearly $\alpha_1...\alpha_r\vand(\alpha_1,...,\alpha_r)$, if $\vand(...)$ is the Vandermonde determinant, and a simple computation gives us that the first is
\[
	(-1)^{i-1}\sum_{j=1}^r (-1)^ja_jM_{ij},
\]
where $M_{ij}$ is the determinant
\[
	M_{ij}=\left|\begin{array}{cccccc}
	\alpha_1 & \hdots & \alpha_{i-1} & \alpha_{i+1} & \hdots & \alpha_r \\
	\vdots & & \vdots & \vdots & & \vdots \\
	\alpha_1^{j-1} & \hdots & \alpha_{i-1}^{j-1} & \alpha_{i+1}^{j-1} & \hdots & \alpha_r^{j-1} \\
	\alpha_1^{j+1} & \hdots & \alpha_{i-1}^{j+1} & \alpha_{i+1}^{j+1} & \hdots & \alpha_r^{j+1} \\
	\vdots & & \vdots & \vdots & & \vdots \\
	\alpha_1^r & \hdots & \alpha_{i-1}^r & \alpha_{i+1}^r & \hdots & \alpha_r^r\end{array}\right|,
\]
that is, the Vandermonde determinant with the $i$-th column and the $j$-th row removed.

With this, equation \eqref{e:lambda-cramer} becomes
\[
	\lambda_i=(-1)^{i-1}\frac{\displaystyle\sum_{j=1}^r (-1)^j a_j M_{ij}}{\alpha_1...\alpha_r\vand(\alpha_1,...,\alpha_r)},
\]
so
\[
	{\bm\lambda}^{\bm\beta}=\prod_{i=1}^r\left((-1)^{i-1}\frac{\displaystyle\sum_{j=1}^r (-1)^j a_j M_{ij}}{\alpha_1...\alpha_r\vand(\alpha_1,...,\alpha_r)}\right)^{\beta_i}
\]
and the numerator of \eqref{e:phi-m-fraction} is
\begin{equation}\label{e:numerator}
	\sum_{|\bm\beta|=m}\binom{m}{\bm\beta}\left(\prod_{i=1}^r\left((-1)^{i-1}\frac{\displaystyle\sum_{j=1}^r (-1)^j a_j M_{ij}}{\vand(\alpha_1,...,\alpha_r)}\right)^{\beta_i}\right)\underset{\bm\beta'\neq\bm\beta}{\prod_{|\bm\beta'|=m}}(1-{\bm\alpha}^{\bm\beta'}).
\end{equation}

Let us check that this is invariant under permutations of the $\alpha_i$. Clearly, it suffices to check that it is invariant under transpositions, so let $\tau_{\gamma\delta}$, with $\gamma<\delta$, be the trasposition that interchanges $\alpha_{\gamma}$ and $\alpha_{\delta}$.

First, it is trivial to see that
\[
	\tau_{\gamma\delta}(\vand(\alpha_1,...,\alpha_r))=(-1)\vand(\alpha_1,...,\alpha_r).
\]

Second, for the determinants $M_{ij}$, we get that
\[
	\tau_{\gamma\delta}(M_{ij})=(-1)M_{ij}\quad\text{ if }i\neq\gamma,\delta,
\]
and that
\[
	\tau_{\gamma\delta}(M_{\gamma j})=(-1)^{\delta-\gamma-1}M_{\delta j}, \qquad \tau_{\gamma\delta}(M_{\delta j})=(-1)^{\delta-\gamma-1} M_{\gamma j}.
\]

As a consequence, in the product
\[
	{\bm\lambda}^{\bm\beta}{\bm\alpha}^{\bm\beta}=\prod_{i=1}^r\left((-1)^{i-1}\frac{\displaystyle\sum_{j=1}^r (-1)^j a_j M_{ij}}{\vand(\alpha_1,...,\alpha_r)}\right)^{\beta_i},
\]
$\tau_{\gamma\delta}$ does not change the factors with $i\neq\gamma,\delta$, and it interchanges the bases of the factors with $i=\gamma$ and $i=\delta$. Therefore, the effect of $\tau_{\gamma\delta}$ on $(\bm\lambda\bm\alpha)^{\bm\beta}$ is to change
\[
	\bm\beta=(\beta_1,...,\beta_{\gamma},...\beta_{\delta},...\beta_r)
\]
into
\[
	(\beta_1,...,\beta_{\delta},...,\beta_{\gamma},...,\beta_r).
\]
Checking now what happens to each term of the sum in \eqref{e:numerator}, if we denote by $\tau_{\gamma\delta}(\bm\beta)$ the result of interchanging the $\gamma$-th and $\delta$-th components, we get that
\[
	\tau_{\gamma\delta}\left(\binom{m}{\bm\beta}(\bm\lambda\bm\alpha)^{\bm\beta}\underset{\bm\beta'\neq\bm\beta}{\prod_{|\bm\beta'|=m}}(1-{\bm\alpha}^{\bm\beta'})\right)=\binom{m}{\bm\beta}(\bm\lambda\bm\alpha)^{\tau_{\gamma\delta}(\bm\beta)}\underset{\bm\beta'\neq\bm\beta}{\prod_{|\bm\beta'|=m}}(1-{\bm\alpha}^{\tau_{\gamma\delta}(\bm\beta')}).
\]

Since
\[
	\binom{m}{\bm\beta}=\binom{m}{\tau_{\gamma\delta}(\bm\beta)},
\]
we find that the effect of $\tau_{\gamma\delta}$ in the full sum
\[
	\sum_{|\bm\beta|=m}\binom{m}{\bm\beta}(\bm\lambda\bm\alpha)^{\bm\beta}\underset{\bm\beta'\neq\bm\beta}{\prod_{|\bm\beta'|=m}}(1-{\bm\alpha}^{\bm\beta'})
\]
is to just commute its terms, which leaves the sum invariant. Therefore, we get that the numerator in \eqref{e:phi-m-fraction} is also invariant under permutations of the $\alpha_i$, which means it is a rational number.
\end{proof}

%%%%%%%%%%%%%%%%%%%%%%%%%%%%%%%%%%%%%%%%%%%%%%%%%%%%%%%%%%%%%%%%%%%%%%%%%%%%%%%%%%%
\section{Closing remarks}
\label{sec:closing}
%%%%%%%%%%%%%%%%%%%%%%%%%%%%%%%%%%%%%%%%%%%%%%%%%%%%%%%%%%%%%%%%%%%%%%%%%%%%%%%%%%%
%%%%%%%%%%%%%%%%%%%%%%%%%%%%%%%%%%%%%%%

Even if we are mainly interested in studying integer LRS whose minimal polynomial is irreducible, many of our results are still true if we remove some of these restrictions. For example, the proof of our main result about meromorphic continuation, Theorem \ref{t:meromorphic continuation}, can be carried in exactly the same way if, instead of an integer LRS with irreducible minimal polynomial, we relax the hypotheses to allow for a \emph{real recurrence sequence with a dominant root $>1$ with multiplicity $1$} (there can be multiple roots as long as they are not the dominant one). Corollary \ref{c:residues} holds under the same hypotheses.

As for Theorem \ref{t:values-negative-integers} about the values of the zeta function at negative integers which are not poles, once again the same proof holds if we assume the sequence to be only rational, not integer, but still with a separable (not necessarily irreducible) and with a dominant root $>1$.

\end{document}